\documentclass[11pt,oneside]{amsart}
\usepackage{lmodern}
\usepackage{lmodern}
\usepackage[T1]{fontenc}
\usepackage[utf8]{inputenc}
\usepackage{xcolor}
\usepackage{mathtools}
\usepackage{enumitem}
\usepackage{amstext}
\usepackage{amsthm}
\usepackage{amssymb}
\usepackage[a4paper]{geometry}
\usepackage{microtype}
\usepackage[bookmarks=false,
 breaklinks=false,pdfborder={0 0 1},backref=false,colorlinks=true]
 {hyperref}
\hypersetup{
 linkcolor=blue,citecolor=blue,urlcolor=blue}

\makeatletter
\numberwithin{equation}{section}
\numberwithin{figure}{section}
\usepackage{amsthm}\usepackage{mathrsfs}

\newcommand{\R}{\mathbb{R}}

\newcommand{\avg}[1]{\left\langle #1\right\rangle}
\newcommand{\AS}{\mathcal{A}_{\mathcal S}}

\newcommand{\dd}{d}

\makeatother

\providecommand\theoremname{Theorem}
\theoremstyle{plain}
\newtheorem{thm}{\protect\theoremname}
\providecommand\lemmaname{Lemma}
\newtheorem{lem}[thm]{\protect\lemmaname}
\providecommand\propositionname{Proposition}
\newtheorem{prop}[thm]{\protect\propositionname}
\providecommand\corollaryname{Corollary}
\newtheorem{cor}[thm]{\protect\corollaryname}
\begin{document}
\title[Entropy and Orlicz endpoint estimates and the dual M-W estimate]{On Entropy and Orlicz endpoint estimates and the dual Muckenhoupt-Wheeden
estimate}
\author{Israel P. Rivera-Ríos}
\address{Departamento de Análisis Matemático, Estadística e Investigación Operativa
y Matemática Aplicada. Facultad de Ciencias. Universidad de Málaga
(Málaga, Spain).}
\email{israelpriverarios@uma.es}
\begin{abstract}
In this paper the incomparability of certain Rahm entropy maximal
functions introduced by Rahm \cite{Ra} and Orlicz maximal functions
that arise in quantitative endpoint estimates for Calderón-Zygmund
operators is established. Also the dual Muckenhoupt-Wheeden two-weight
estimates introduced by Lerner, Ombrosi and Pérez \cite{LOP} are
revisited, improving the possible choices for the maximal function
in the left hand side of the inequality, via a transference principle
from two-weight endpoint inequalities.
\end{abstract}

\maketitle

\section{Introduction and main results}\label{sec:IMR}

We say that $w$ is a weight if it is a non-negative, locally integrable
function. Quantitative weighted estimates have been a fruitful field
of study in the last fifteen years, having as a cornerstone the nowadays
called $A_{2}$ theorem \cite{H} which states that if $T$ is a Calderón-Zygmund
operator then for $1<p<\infty$
\[
\|Tf\|_{L^{p}(w)}\leq c_{T,p}[w]^{\max\left\{ 1,\frac{1}{p-1}\right\} }_{A_{p}}\|f\|_{L^{p}(w)}
\]
where $[w]_{A_{p}}=\sup_{Q}\left(\frac{1}{|Q|}\int_{Q}w\right)\left(\frac{1}{|Q|}\int_{Q}w^{-\frac{1}{p-1}}\right)^{p-1}<\infty$
where the supremum is taken over cubes. This estimate is sharp in
the sense that the exponents for which it holds for each $w\in A_{p}$
cannot be replaced by smaller ones.

In the case $p=1$, Lerner, Ombrosi and Pérez \cite{LOPA1} showed
that 
\begin{equation}
\sup_{t>0}tw\left(\left\{ x\in\mathbb{R}^{n}:|Tf(x)|>t\right\} \right)=\|Tf\|_{L^{1,\infty}(w)}\leq c_{T}[w]_{A_{1}}\log(e+[w]_{A_{1}})\|f\|_{L^{1}(w)}.\label{eq:w11Sharp}
\end{equation}
Later on, Lerner, Nazarov and Ombrosi \cite{LNO} showed that actually
this estimate is sharp.

A type of inequalities related to \eqref{eq:w11Sharp}, that have
attracted the attention of a number of researchers are two weight
inequalities for pairs of weights $(w,u)$ 
\[
\|Tf\|_{L^{1,\infty}(w)}\leq c_{T}\|f\|_{L^{1}(u)}
\]
where $u=\tilde{M}w$ for some suitable maximal operator $\tilde{M}$.
The motivation to study this kind of inequality comes from the seminal
work of Fefferman and Stein \cite{FS} in which they showed that
\begin{equation}
\|Mf\|_{L^{1,\infty}(w)}\leq c_{n}\|f\|_{L^{1}(Mw)}.\label{eq:FS}
\end{equation}
That gave rise to the so called Muckenhoupt-Wheeden conjecture, namely,
does \eqref{eq:FS} hold with $M$ replaced by the Hilbert transform?
Such a question was answered in the negative by Reguera \cite{R}
for dyadic models and also by the same author and Thiele \cite{RT}
for the Hilbert transform. Then it became a natural question to consider
suitable maximal operators larger than $M$ such that the inequality
holds. Such a direction was, probably first, pursued by Pérez \cite{P},
who showed that 
\begin{equation}
\|Tf\|_{L^{1,\infty}(w)}\leq c_{T,\varepsilon}\|f\|_{L^{1}(M_{L(\log L)^{\varepsilon}}w)}.\label{eq:PerezLogepsilon}
\end{equation}
Now we provide some notation to make the equation above understandable.
If $\Phi$ is a Young function, namely $\Phi:[0,\infty)\rightarrow[0,\infty)$
is convex with $\Phi(0)=0$, we define the localized Luxemburg norm
\[
\|h\|_{\Phi,Q}:=\inf\left\{ \lambda>0:\frac{1}{|Q|}\int_{Q}\Phi\!\left(\frac{|h(x)|}{\lambda}\right)dx\leq1\right\} 
\]
and the corresponding Orlicz maximal operator 
\[
M_{\Phi}h(x):=\sup_{Q\ni x}\|h\|_{\Phi,Q}
\]
where supremum is taken over cubes with sides parallel to the coordinate
axes. We denote by $\overline{\Phi}$ the complementary Young function,
\[
\overline{\Phi}(s):=\sup_{t\geq0}\{st-\Phi(t)\}.
\]
Coming back to our discussion, in \eqref{eq:PerezLogepsilon} $M_{L(\log L)^{\varepsilon}}=M_{\Phi}$
with $\Phi(t)=t\log(e+t)^{\varepsilon}$. Later on, relying upon sparse
domination techniques Domingo-Salazar, Lacey and Rey \cite{DSLR}
showed that for every weight $w$ and every $f\in L^{1}(\R^{n})$,
\begin{equation}
\|Tf\|_{L^{1,\infty}(w)}\leq C_{n,T}\bigl(1+c_{\Phi}\bigr)\|f\|_{L^{1}(M_{\Phi}w)}.\label{eq:DSLR}
\end{equation}
where

\begin{equation}
c_{\Phi}:=\sum^{\infty}_{k=1}\frac{1}{\overline{\Phi}^{-1}(2^{2^{k}})}.\label{eq:cPhi}
\end{equation}
Note that if we set
\[
L_{1}(t):=\log(e+t),\qquad L_{j+1}(t):=\log(e+L_{j}(t)),\quad j\geq1
\]
\eqref{eq:DSLR} holds, for instance for $\Phi(t)=t\,L_{2}(t)^{1+\varepsilon}$,
with $\varepsilon>0$. Note that, in the case of interest $0<\varepsilon\leq1$,
$c_{\Phi}\simeq\frac{1}{\varepsilon}$ for this choice of $\Phi$.

Related to this problem, Caldarelli, Lerner and Ombrosi \cite{CLO}
proved that \eqref{eq:DSLR} does not hold if 
\[
\lim_{t\rightarrow\infty}\frac{\Phi(t)}{tL_{2}(t)}=0.
\]
Hence the remaining open problem is to show whether the inequality
holds in the case $\Phi(t)=tL_{2}(t)$.

A different way to consider larger maximal functions to balance the
inequality was introduced by Rahm \cite{Ra} following the entropy
bump approach that had been begun by Treil and Volberg \cite{TV}
and simplified later on in works such as \cite{LS,RS2,RS1}. Let us
recall Rahm's framework. Let $\varepsilon:[1,\infty]\rightarrow[1,\infty]$
be an increasing function with 
\[
\kappa_{\varepsilon}=\sum^{\infty}_{k=-1}\frac{1}{\varepsilon\left(2^{2^{k}}\right)}<\infty.
\]
A possible choice for $\varepsilon$ is $\varepsilon_{\delta}(t)=L_{2}(t)L_{3}(t)^{1+\delta}$
with $\delta>0$. For any cube $Q$ we define 
\[
\rho_{w}(Q)=\frac{1}{w(Q)}\int_{Q}M(\chi_{Q}w)(x)dx
\]
where $M=M_{L}$ is the usual uncentered Hardy-Littlewood maximal
function. If $w(Q)=0$, then $\rho_{w}(Q)=1$. We then define the
following entropy bump maximal function
\[
M_{\varepsilon}w=\sup_{x\in Q}\frac{w(Q)}{|Q|}L_{1}(\rho_{w}(Q))\varepsilon(\rho_{w}(Q)).
\]
For this kind of maximal functions Rahm shows that
\begin{equation}
\|Tf\|_{L^{1,\infty}(w)}\leq C_{n,T}\kappa_{\varepsilon}\|f\|_{L^{1}(M_{\varepsilon}w).}\label{eq:Rahm}
\end{equation}
Rahm points out that ``borderline'' conditions do not seem comparable,
namely
\[
M_{\varepsilon_{\delta}}w(x)\not\simeq M_{\Phi_{1+\rho}}w(x).
\]
where 
\begin{align*}
\Phi_{1+\rho}(t) & =tL_{2}(t)L_{3}(t)^{1+\rho}\\
 & =t\log(e+\log(e+t))[\log(e+\log(e+\log(e+t)))]^{1+\rho}
\end{align*}
Our first results confirm that guess.
\begin{thm}
\label{thm:RahmInfty}There exists a weight $w$ such that for every
$\delta>0$ and every $\rho>0$
\[
M_{\varepsilon_{\delta}}w(x)=\infty
\]
and
\[
M_{\Phi_{1+\rho}}w(x)<\infty
\]
for a set of positive measure. From this it readily follows that there
is no constant $C>0$ such that 
\[
M_{\varepsilon_{\delta}}w(x)\leq CM_{\Phi_{1+\rho}}w(x)
\]
for every $x\in\mathbb{R}$ and every weight $w$.
\end{thm}

\begin{thm}
\label{thm:RahmBlowsSlower}There exists a family of weights $\{w_{N}\}$
such that for every $\delta>0$ and every $\rho>0$, 
\[
M_{\Phi_{1+\rho}}w_{N}(x)\simeq_{\rho}L_{2}(N)L_{3}(N)^{1+\rho},
\]
and 
\[
M_{\varepsilon_{\delta}}w_{N}(x)\simeq_{\delta}L_{2}(N)L_{3}(N)L_{4}(N)^{1+\delta}
\]
in a set of positive measure $J$. Consequently there is no constant
$C$ such that 
\[
M_{\Phi_{1+\rho}}w(x)\leq CM_{\varepsilon_{\delta}}w(x)
\]
for every $x\in\mathbb{R}$ and every weight $w$.
\end{thm}

Now we turn our attention to our next contribution. Lerner, Ombrosi
and Pérez showed in \cite{LOP} that for every $\varepsilon>0$, 
\[
\left\Vert \frac{Tf}{M_{L(\log L)^{1+\varepsilon}}w}\right\Vert _{L^{1,\infty}(w)}\leq c_{T,\varepsilon}\|f\|_{L^{1}}.
\]
Their motivation to pursue such an inequality at that time was that
if the Muckenhoupt-Wheeden conjecture was true then the extrapolation
theorem from \cite{CUP}, would allow them to show that, let's say
for the Hilbert transform,
\[
\int_{\mathbb{R}}|Hf|^{p}w\leq c\int_{\mathbb{R}}|f|^{p}\left(\frac{Mw}{w}\right)^{p}wdx\qquad1<p<\infty
\]
which by duality would imply that for every $1<p<\infty$
\begin{equation}
\int_{\mathbb{R}}\left(\frac{|Hf|}{Mw}\right)^{p'}w\leq c\int_{\mathbb{R}}\left(\frac{|f|}{w}\right)^{p'}wdx\label{eq:DualizedLOP}
\end{equation}
and hence 
\[
\left\Vert \frac{Hf}{Mw}\right\Vert _{L^{1,\infty}(w)}\leq c_{T}\|f\|_{L^{1}}
\]
would be a limiting case of \eqref{eq:DualizedLOP}.

Quite recently Osękowski \cite{O} settled the following result. If
$T$ is a Calderón-Zygmund operator, then if $w\in A_{1}$, 
\[
\left\Vert \frac{Tf}{w}\right\Vert _{L^{1,\infty}(w)}\lesssim[w]_{A_{1}}\|f\|_{L^{1}}
\]
He also provided a direct example showing that for arbitrary weights
\[
\left\Vert \frac{Hf}{M_{\Phi}w}\right\Vert _{L^{1,\infty}(w)}\lesssim\|f\|_{L^{1}}
\]
does not hold if $\lim_{t\rightarrow\infty}\frac{\Phi(t)}{t\log(\log(e^{e}+t))}=0.$

A key to \eqref{eq:DSLR} and \eqref{eq:Rahm} inequalities is sparse
domination. Recall that a dyadic lattice $\mathcal{D}$ is a family
of cubes of $\mathbb{R}^{n}$ having the following properties. 
\begin{enumerate}
\item If $Q\in\mathcal{D}$ then also $\mathcal{D}(Q)\subset\mathcal{D}$
where $\mathcal{D}(Q)$ is the family of all the dyadic descendants
of $Q$.
\item If $Q,Q'\in\mathcal{D}$ they have a common ancestor, namely, there
exists $P\in\mathcal{D}.$ such that $Q,Q'\in\mathcal{D}(P)$.
\item If $K\subset\mathbb{R}^{n}$ is a compact set there exists some $Q\in\mathcal{D}$
such that $K\subset Q$.
\end{enumerate}
Given $\eta\in(0,1)$ and a dyadic lattice $\mathcal{D}$ we say that
$\mathcal{S}\subset\mathcal{D}$ is a $\eta$-sparse family if for
every $Q\in\mathcal{S}$ there exists a subset $E_{Q}\subset Q$ such
that $\eta|Q|\leq|E_{Q}|$ and the $E_{Q}$ are pairwise disjoint.
For a further insight on sparse families we remit the interested reader
to \cite{LN}. Finally if $\mathcal{S}$ is a family of cubes and
$f$ is a locally integrable function we define the sparse operator
$A_{\mathcal{S}}$ as
\[
A_{\mathcal{S}}f(x)=\sum_{Q\in\mathcal{S}}\frac{1}{|Q|}\int_{Q}f\chi_{Q}(x).
\]
Actually \eqref{eq:DSLR} and \eqref{eq:Rahm} were obtained for sparse
operators with uniform constant for every sparse family. 

Our main result is the following.
\begin{thm}
\label{thm:Transf} Let $T$ be a Calderón-Zygmund operator, $w\geq0$
a weight and $u$ be a measurable function such that $0<u<\infty$
$w$-almost everywhere. Let $\eta\in(0,1)$. Assume that there exists
a constant $\kappa_{u,w,\eta}>0$ depending on $u$, $w$ and $\eta$
such that for every non-negative $g\in L_{1}(u)$
\[
\|A_{\mathcal{S}}g\|_{L^{1,\infty}(w)}\leq\kappa_{u,w,\eta}\|g\|_{L^{1}(u)}
\]
uniformly for every $\eta-$sparse family $\mathcal{S}\subset\mathcal{D}$
where $\mathcal{D}$ is any dyadic lattice. Then, for every $f\geq0$,
\[
\left\Vert \frac{Tf}{u}\right\Vert _{L^{1,\infty}(w)}\lesssim_{n,T,\eta}\kappa_{u,w,\eta}\|f\|_{L^{1}}.
\]
\end{thm}

As we will see in \ref{subsec:CorollariesTP} we have that in particular
this result implies estimates such as
\begin{align*}
\left\Vert \frac{Tf}{M_{L(\log\log L)(\log\log\log L)^{1+\varepsilon}}w}\right\Vert _{L^{1,\infty}(w)} & \lesssim\frac{1}{\varepsilon}\|f\|_{L^{1}}\\
\left\Vert \frac{Tf}{M_{L(\log\log L)^{1+\varepsilon}}w}\right\Vert _{L^{1,\infty}(w)} & \lesssim\frac{1}{\varepsilon}\|f\|_{L^{1}}.
\end{align*}
Consequently, it remains an open question whether 
\[
\left\Vert \frac{Tf}{M_{L(\log\log L)}w}\right\Vert _{L^{1,\infty}(w)}\lesssim\|f\|_{L^{1}}
\]
is true or not. 

At this point it is worth noting that Theorem \ref{thm:Transf} holds
as well for maximal Calderón-Zygmund operators since Theorem \ref{thm:Transf}
is actually a direct corollary of a result with a similar statement
that holds for sparse operators. We provide further details in Subsection
\ref{subsec:CorollariesTP}.

The remainder of the paper is organized as follows. In Section \ref{sec:IncompMaxFunct}
we present the proofs of the incomparability results for the Orlicz
and the entropy Rahm maximal functions together with some lemmata
required for them. Section \ref{sec:TP+corollaries} is devoted to
establish Theorem \ref{thm:Transf}.

\section{The incomparability of the Orlicz Maximal functions and the Entropy
Rahm Maximal Functions}\label{sec:IncompMaxFunct}

\subsection{Some notation and technical lemmata}

Recall that as we mentioned a few lines above, we shall call
\[
L_{1}(t)=\log(e+t),\qquad L_{j+1}(t)=\log(e+L_{j}(t)),\quad j\geq1.
\]
Accordingly we consider
\[
\Phi_{\alpha}(t)=tA_{\alpha}(t),\qquad A_{\alpha}(t)=L_{2}(t)L_{3}(t)^{\alpha}.
\]
For the entropy bump we choose 
\[
\varepsilon_{\delta}(s)=L_{2}(s)L_{3}(s)^{1+\delta},\qquad\delta>0,
\]
and we define 
\[
B_{\delta}(s)=L_{1}(s)\varepsilon_{\delta}(s)=L_{1}(s)L_{2}(s)L_{3}(s)^{1+\delta}.
\]
With the definitions above at our disposal we are going to be dealing
with the comparability of 
\[
M_{\Phi_{\alpha}}w(x)=\sup_{x\in I}\|w\|_{\Phi_{\alpha},I}
\]
 {}
\[
M_{\varepsilon_{\delta}}w(x)=\sup_{x\in I}\avg{w}_{I}B_{\delta}(\rho_{w}(I)).
\]
 For notational convenience we introduce as well the following function
\[
H_{\delta}(t):=L_{2}(t)L_{3}(t)L_{4}(t)^{1+\delta}.
\]
Our first technical lemma shows that multiplying the argument of a
function built upon $L_{i}$ functions as the ones above by a fixed
positive constant changes essentially by a multiplicative factor its
value.
\begin{lem}
\label{lem:slow} Let 
\[
F(t)=\prod^{m}_{j=1}L_{j}(t)^{a_{j}},\qquad a_{j}\geq0,
\]
with $m<\infty$. For every $\eta>0$ there is a constant $C=C(F,\eta)$
such that 
\[
\frac{F(s)}{F(t)}\leq C\left(\frac{s}{t}\right)^{\eta},\qquad s\geq t\geq2.
\]
In particular, multiplication of the argument by a fixed positive
constant changes $A_{\alpha}$, $B_{\delta}$, or $H_{\delta}$ only
by a multiplicative constant for large arguments. 
\end{lem}

\begin{proof}
Write $s=qt$ with $q\geq1$. Since 
\[
e+s=e+qt\leq q(e+t),
\]
we have 
\begin{equation}
L_{1}(s)\leq L_{1}(t)+\log q.\label{eq:step1}
\end{equation}
The map $u\mapsto\log(e+u)$ has derivative at most $1/e<1$. Hence,
inductively, 
\begin{equation}
0\leq L_{j}(s)-L_{j}(t)\leq\log q,\qquad j\geq1.\label{eq:step2}
\end{equation}
Indeed, note that for $j=1$ the assertion is \eqref{eq:step1}. Assume
the inequalities hold for a certain $j$. Then we have that 
\[
L_{j+1}(s)-L_{j+1}(t)=\log(e+L_{j}(s))-\log(e+L_{j}(t))
\]
Since $L_{j}(s)-L_{j}(t)\geq0$, by the mean value theorem, there
exists $\tau\in[L_{j}(t),L_{j}(s)]$ such that
\[
\log(e+L_{j}(s))-\log(e+L_{j}(t))=\frac{1}{e+\tau}\left(L_{j}(s)-L_{j}(t)\right)
\]
and \eqref{eq:step2} readily follows from the inductive hypothesis.

Since $L_{j}(t)\geq L_{j}(2)=c_{j}>0$ for $t\geq2$, for some $C_{j}>0$,
\[
\frac{L_{j}(s)}{L_{j}(t)}\leq C_{j}(1+\log q).
\]
Calling $A=\sum_{j}a_{j}$, then 
\[
\frac{F(s)}{F(t)}\leq C_{F}(1+\log q)^{A}.
\]
Now, for every $\eta>0$, 
\[
(1+\log q)^{A}\leq C_{A,\eta}q^{\eta},\qquad q\geq1,
\]
which proves the desired estimate.
\end{proof}

Our second technical Lemma presents the relation between $H_{\delta}$
and $B_{\delta}$. This result will be crucial in the sequel.
\begin{lem}
\label{lem:shift} Let $\delta>0$. If 
\begin{align*}
H_{\delta}(t) & =L_{2}(t)L_{3}(t)L_{4}(t)^{1+\delta}\\
B_{\delta}(s) & =L_{1}(s)L_{2}(s)L_{3}(s)^{1+\delta}
\end{align*}
 then, for every fixed $C>0$ and $r\geq2$, 
\[
B_{\delta}\bigl(C(1+\log r)\bigr)\simeq_{C,\delta}H_{\delta}(r).
\]
\end{lem}

\begin{proof}
For $r\geq2$, 
\[
L_{1}\bigl(C(1+\log r)\bigr)\simeq L_{2}(r),
\]
and, after applying $u\mapsto\log(e+u)$ successively, 
\[
L_{2}\bigl(C(1+\log r)\bigr)\simeq L_{3}(r),\qquad L_{3}\bigl(C(1+\log r)\bigr)\simeq L_{4}(r).
\]
Multiplying these three equivalences gives the desired equivalence.
\end{proof}

Our next result provides information on computing the inverse of $\Phi_{\alpha}$.
\begin{lem}
\label{lem:inverse} Let $\Phi_{\alpha}(r)=rL_{2}(r)L_{3}(r)^{\alpha}$.
For $r$ sufficiently large, 
\[
\Phi^{-1}_{\alpha}(r)\simeq_{\alpha}\frac{r}{L_{2}(r)L_{3}(r)^{\alpha}}.
\]
\end{lem}

\begin{proof}
First we note that
\begin{align*}
\Phi_{\alpha}\left(\frac{r}{L_{2}(r)L_{3}(r)^{\alpha}}\right) & =\frac{r}{L_{2}(r)L_{3}(r)^{\alpha}}L_{2}\left(\frac{r}{L_{2}(r)L_{3}(r)^{\alpha}}\right)L_{3}\left(\frac{r}{L_{2}(r)L_{3}(r)^{\alpha}}\right)^{\alpha}\\
 & \leq\frac{r}{L_{2}(r)L_{3}(r)^{\alpha}}L_{2}\left(r\right)L_{3}\left(r\right)^{\alpha}=r
\end{align*}
so 
\[
\frac{r}{L_{2}(r)L_{3}(r)^{\alpha}}\leq\Phi^{-1}_{\alpha}(r).
\]
Now we note that
\[
\lim_{r\rightarrow\infty}\frac{L_{2}\left(\frac{r}{L_{2}(r)L_{3}(r)^{\alpha}}\right)}{L_{2}(r)}=1
\]
and $L_{2}\left(\frac{r}{L_{2}(r)L_{3}(r)^{\alpha}}\right)<L_{2}\left(r\right)$
for every $r>1$, then there exists $r_{0,\alpha}>1$ such that for
every $r>r_{0,\alpha}$, 
\[
\frac{L_{2}\left(\frac{r}{L_{2}(r)L_{3}(r)^{\alpha}}\right)}{L_{2}(r)}>\frac{1}{2}.
\]
Analogously there exists $r_{1,\alpha}>1$ such that for every $r>r_{1,\alpha}$
\[
\left(\frac{L_{3}\left(\frac{r}{L_{2}(r)L_{3}(r)^{\alpha}}\right)}{L_{3}(r)}\right)^{\alpha}\geq\frac{1}{2}.
\]
Consequently for $r>\max\{r_{1,\alpha},r_{0,\alpha}\}$, 
\[
\Phi_{\alpha}\left(\frac{r}{L_{2}(r)L_{3}(r)^{\alpha}}\right)\geq\frac{1}{4}r.
\]
Now evaluating in $4r$
\[
\Phi_{\alpha}\left(\frac{4r}{L_{2}(4r)L_{3}(4r)^{\alpha}}\right)\geq r
\]
and that leads us to
\[
\Phi^{-1}_{\alpha}(r)\leq\frac{4r}{L_{2}(4r)L_{3}(4r)^{\alpha}}\leq c_{\alpha}\frac{r}{L_{2}(r)L_{3}(r)^{\alpha}}.
\]
\end{proof}

Before presenting our last lemma, we recall the well known fact that
if $F=[a,b]$ is an interval then
\begin{equation}
M\chi_{F}(x)=\begin{cases}
1, & x\in F,\\[1mm]
{\displaystyle \frac{|F|}{|F|+\operatorname{dist}(x,F)},} & x\notin F,
\end{cases}\label{eq:MchiInt}
\end{equation}
where $M$ stands for the standard uncentered maximal function.
\begin{lem}
\label{lem:lemIndicator}Let $I$ be a finite interval. Let $F\subset I$
be a subinterval, and let $v=A\chi_{F}$ with $A>0$. Write $s_{-}$
and $s_{+}$ for the lengths of the two components of $I\setminus F$
(one of them could be zero), so that 
\[
s_{-}+s_{+}=|I|-|F|.
\]
Then 
\[
\rho_{v}(I)=1+\log\!\left(1+\frac{s_{-}}{|F|}\right)+\log\!\left(1+\frac{s_{+}}{|F|}\right).
\]
Consequently, if 
\[
r=\frac{|I|}{|F|}\geq1,
\]
then 
\begin{equation}
\rho_{v}(I)\simeq1+\log r.\label{eq:rhov(I)}
\end{equation}
with absolute constants, uniformly in the position of $F$ inside
$I$ and in $A$. 
\end{lem}

\begin{proof}
Note that the scalar $A$ cancels out in the definition of $\rho_{v}$.
Hence we have that,
\[
\rho_{v}(I)=\frac{1}{|F|}\int_{I}M\chi_{F}(x)\,dx.
\]
Without loss of generality we may assume that $I=[0,l]$. Note that
then 
\[
\frac{1}{|F|}\int_{I}M\chi_{F}(x)\,dx=\frac{1}{|F|}\int_{F}M\chi_{F}(x)\,dx+\frac{1}{|F|}\int_{I\setminus F}M\chi_{F}(x)\,dx
\]
By \eqref{eq:MchiInt}, we have that
\[
\frac{1}{|F|}\int_{F}M\chi_{F}(x)\,dx=1.
\]
Now we note that there are two possible situations for $I\setminus F$. 
\begin{itemize}
\item $I\setminus F=J_{1}\cup J_{2}$ where $J_{i}$ are non empty intervals.
In this case we call $s_{-}$ the length of the leftmost one, let
us say $J_{1}$, and $s_{+}$ the length of the rightmost one, $J_{2}$.
\item $I\setminus F=J$ where $J$ is an interval 
\begin{itemize}
\item to the left of $F$. In this case $s_{-}>0$ and $s_{+}=0$. 
\item or to the right of $F$. In this case $s_{+}>0$ and $s_{-}=0$. 
\end{itemize}
\end{itemize}
Note that again by \eqref{eq:MchiInt}, if $x\in J_{1}$ 
\[
M\chi_{F}(x)=\frac{|F|}{|F|+\operatorname{dist}(x,F)}=\frac{|F|}{|F|+(s_{-}-x)}.
\]
Consequently
\begin{align*}
\frac{1}{|F|}\int_{J_{1}}M\chi_{F}(x)\,dx & =\frac{1}{|F|}\int^{s_{-}}_{0}\frac{|F|}{|F|+(s_{-}-x)}\,dx\\
 & =\left[-\log\!\left(|F|+(s_{-}-x)\right)\right]^{s_{-}}_{0}\\
 & =\log\left(\frac{|F|+s_{-}}{|F|}\right)=\log\left(1+\frac{s_{-}}{|F|}\right)
\end{align*}
Analogously, if $x\in J_{2}$ then 
\[
M\chi_{F}(x)=\frac{|F|}{|F|+\operatorname{dist}(x,F)}=\frac{|F|}{|F|+x-(|F|+s_{-})}=\frac{|F|}{x-s_{-}}.
\]
This yields
\begin{align*}
\frac{1}{|F|}\int_{J_{2}}M\chi_{F}(x)\,dx & =\frac{1}{|F|}\int^{|F|+s_{-}+s_{+}}_{|F|+s_{-}}\frac{|F|}{x-s_{-}}dx=\left[\log\left(x-s_{-}\right)\right]^{|F|+s_{-}+s_{+}}_{|F|+s_{-}}\\
 & =\log\left(\frac{|F|+s_{+}}{|F|}\right)=\log\left(1+\frac{s_{+}}{|F|}\right)
\end{align*}
Combining the computations above,
\[
\frac{1}{|F|}\int_{I}M\chi_{F}(x)\,dx=1+\log\!\left(1+\frac{s_{-}}{|F|}\right)+\log\!\left(1+\frac{s_{+}}{|F|}\right)
\]
as we wanted to show.

Finally for the upper bound in \eqref{eq:rhov(I)}, we observe that
\[
\rho_{v}(I)\leq1+2\log\!\left(1+\frac{s_{-}+s_{+}}{|F|}\right)\leq1+2\log r.
\]
For the lower bound, one of $s_{-},s_{+}$ is at least $(|I|-|F|)/2$.
Hence 
\[
\rho_{v}(I)\geq1+\log\!\left(1+\frac{|I|-|F|}{2|F|}\right).
\]
If $r\geq2$, the right-hand side is comparable to $1+\log r$; if
$1\leq r\leq2$, both quantities are comparable to $1$. This proves
\eqref{eq:rhov(I)}. 
\end{proof}

\subsection{A weight with finite Orlicz maximal function, infinite entropy maximal
function}

The remainder of this section is devoted to prove the theorem above.
For that purpose we begin setting 
\[
a=e^{-4}
\]
and defining 
\begin{equation}
w(x)=\frac{\chi_{(0,a)}(x)}{x(\log(1/x))^{2}}.\label{eq:weightOrliczGood}
\end{equation}
Let 
\[
Q_{0}=(0,1),\qquad J=\left(\frac{1}{2},\frac{3}{4}\right).
\]

It is not hard to check that
\[
\int^{a}_{0}w(x)\,dx=\frac{1}{4}.
\]
Consequently, $w$ is locally integrable. We are going to show the
following result
\begin{thm}
\label{thm:finiteOrlInfEnt}For every $\alpha,\delta>0$, 
\[
M_{\Phi_{\alpha}}w(x)<\infty
\]
and 
\[
M_{\varepsilon_{\delta}}w(x)=\infty
\]
for every $x\in J$.
\end{thm}

We devote the remainder of the section to settle the result above
which in turn implies Theorem \ref{thm:RahmInfty}.
\begin{lem}
\label{lem:first-orlicz-int} For every fixed $\alpha>0$, 
\[
\int^{a}_{0}\Phi_{\alpha}(w(x))\,dx<\infty.
\]
\end{lem}

\begin{proof}
Note that if $x=e^{-u}$, 
\[
w(x)=\frac{e^{u}}{u^{2}}\qquad u>4.
\]
Then for every $u>4$, 
\[
\log w(x)=u-2\log u\leq u.
\]
It follows that 
\[
L_{2}(w(x))\lesssim\log(e+u),\qquad L_{3}(w(x))\lesssim\log\log(e^{e}+u).
\]
Thus 
\[
\Phi_{\alpha}(w(x))\,\lesssim e^{u}\frac{\log(e+u)\,[\log\log(e^{e}+u)]^{\alpha}}{u^{2}}
\]
Note that then
\[
\int^{a}_{0}\Phi_{\alpha}(w(x))\,dx=\int^{\infty}_{4}e^{-u}\Phi_{\alpha}(w(e^{-u}))\,du\lesssim\int^{\infty}_{4}\frac{\log(e+u)\,[\log\log(e^{e}+u)]^{\alpha}}{u^{2}}du<\infty
\]
as we wanted to show.
\end{proof}

\begin{prop}
\label{prop:first-orlicz} For every $\alpha>0$ there is $C_{\alpha}<\infty$
such that 
\[
M_{\Phi_{\alpha}}w(x)\leq C_{\alpha},\qquad x\in J.
\]
\end{prop}

\begin{proof}
Let
\[
d_{0}=\frac{1}{2}-a>0.
\]
 We have to prove that there exists $C_{\alpha}>0$ such that for
every interval $I$ with $I\cap J\not=\emptyset$, then 
\[
\frac{1}{|I|}\int_{I}\Phi_{\alpha}\!\left(\frac{w}{C_{\alpha}}\right)\,dx\leq1.
\]
Note that if $I\cap(0,a)=\emptyset$, then $\|w\|_{\Phi_{\alpha},I}=0$.
If $I$ meets $(0,a)$, then necessarily 
\begin{equation}
|I|\geq d_{0}.\label{eq:SizeI}
\end{equation}
By Lemma~\ref{lem:first-orlicz-int}, $\Phi_{\alpha}(w)\in L^{1}(0,a)$.
For $\lambda\geq1$, 
\begin{equation}
0\leq\Phi_{\alpha}(w/\lambda)\leq\Phi_{\alpha}(w),\label{eq:Phi(wlambda)Phi(w)}
\end{equation}
and $\Phi_{\alpha}(w/\lambda)\to0$ pointwise as $\lambda\to\infty$.
Then, from dominated convergence theorem it readily follows that
\[
\int^{a}_{0}\Phi_{\alpha}\!\left(\frac{w}{\lambda}\right)\,dx\longrightarrow0.
\]
We may choose $C_{\alpha}$ so large that 
\[
\frac{1}{d_{0}}\int^{a}_{0}\Phi_{\alpha}\!\left(\frac{w}{C_{\alpha}}\right)\,dx\leq1.
\]
Note that then, taking into account \eqref{eq:Phi(wlambda)Phi(w)},
that $I\cap(0,a)\subset(0,a)$ and \eqref{eq:SizeI}, 
\[
\frac{1}{|I|}\int_{I}\Phi_{\alpha}\!\left(\frac{w}{C_{\alpha}}\right)\,dx\leq\frac{1}{d_{0}}\int^{a}_{0}\Phi_{\alpha}\!\left(\frac{w}{C_{\alpha}}\right)\,dx\leq1.
\]
Therefore $\|w\|_{\Phi_{\alpha},I}\leq C_{\alpha}$. This ends the
proof.
\end{proof}

\begin{prop}
\label{prop:first-ent} For the weight in \eqref{eq:weightOrliczGood},
\[
\rho_{w}(Q_{0})=\infty.
\]
Consequently 
\[
M_{\varepsilon_{\delta}}w(x)=\infty,\qquad x\in Q_{0},
\]
and in particular for every $x\in J$.
\end{prop}

\begin{proof}
For $0<x<a$, the interval $(0,x)$ is considered when computing $M(w\chi_{Q_{0}})(x)$.
Hence
\begin{align*}
M(w\chi_{Q_{0}})(x) & \geq\frac{1}{x}\int^{x}_{0}\frac{1}{y\left(\log\left(\frac{1}{y}\right)\right)^{2}}dy=\frac{1}{x}\frac{1}{\log\left(\frac{1}{x}\right)}
\end{align*}
Consequently,
\[
\int_{Q_{0}}M(w\chi_{Q_{0}})(x)\,dx\geq\int^{a}_{0}\frac{dx}{x\log(1/x)}=\int^{\infty}_{4}\frac{du}{u}=\infty.
\]
Since $w(Q_{0})=w(0,a)=\frac{1}{4}$, we have that $\rho_{w}(Q_{0})=\infty$.
Since $Q_{0}$ is an interval considered computing $M_{\varepsilon_{\delta}}w(x)$
for $x\in Q_{0}$ this ends the proof.
\end{proof}

Relying upon the weight defined above and the propositions settled
above the proof of the theorem is straightforward and hence we omit
it.

\subsection{A family of weights for which Orlicz maximal functions blow up faster
than entropy maximal functions}

Let 
\[
N\geq8,\qquad E_{N}=\left(0,\frac{1}{N}\right),\qquad w_{N}=N\chi_{E_{N}},
\]
and again take 
\[
J:=\left(\frac{1}{2},\frac{3}{4}\right).
\]
For the family of weights $w_{N}$ we are going to prove the following,
which is a more precise statement of Theorem \ref{thm:RahmBlowsSlower}.
\begin{thm}
\label{thm:ThmMRahmContMOrlicz}Let $\alpha>1$ and $\delta>0$. For
every $x\in J$ we have that
\[
M_{\Phi_{\alpha}}w_{N}(x)\simeq_{\alpha}A_{\alpha}(N),\qquad x\in J,
\]
and 
\[
M_{\varepsilon_{\delta}}w_{N}(x)\simeq_{\delta}H_{\delta}(N),\qquad x\in J.
\]
where $A_{\alpha}(s)=L_{2}(s)L_{3}(s)^{\alpha}$ and $H_{\delta}(s)=L_{2}(s)L_{3}(s)L_{4}(s)^{1+\delta}.$
\end{thm}

Note that from this result it follows that there is no constant $C>0$
such that 
\[
M_{\Phi_{\alpha}}w\leq CM_{\varepsilon_{\delta}}w
\]
for every weight $w$ and every $x\in\mathbb{R}$. Indeed, note that
if $x\in J$, and $w=w_{N}$ then 
\[
L_{2}(N)L_{3}(N)^{\alpha}\lesssim L_{2}(N)L_{3}(N)L_{4}(N)^{1+\delta}.
\]
Consequently
\[
\frac{L_{3}(N)^{\alpha-1}}{L_{4}(N)^{1+\delta}}\lesssim1.
\]
However, since $\alpha>1$ 
\[
\frac{L_{3}(N)^{\alpha-1}}{L_{4}(N)^{1+\delta}}\rightarrow\infty,
\]
which is a contradiction.

We devote the remainder of this section to settle Theorem \ref{thm:ThmMRahmContMOrlicz}.
We have to compute $M_{\varepsilon_{\delta}}w_{N}(x)$ and $M_{\Phi_{\alpha}}w_{N}(x)$
for every $x\in J.$ We begin noting that if $x\in J$ and $x\in I$
for some arbitrary interval $I$. If $I\cap E_{N}=\varnothing$, both
its Orlicz and entropy contributions are zero.
\begin{prop}
\label{prop:Trivialstuff}Let $x\in J$ and let $I$ be an interval
containing $x$ such that
\[
|F|=|I\cap E_{N}|>0
\]
Then for $N\geq4$ we have that
\[
|I|>\frac{3}{8}\qquad\frac{|I|}{|F|}\geq\frac{3N}{8}\qquad\avg{w_{N}}_{I}=\frac{N|F|}{|I|}
\]
\end{prop}

\begin{proof}
Since $x\geq1/2$ and $I$ meets $(0,1/N)$, 
\[
|I|\geq\frac{1}{2}-\frac{1}{N}\geq\frac{3}{8}.
\]
On the other hand $|F|\leq1/N$. Therefore 
\[
\frac{|I|}{|F|}\geq\frac{3N}{8}.
\]
Finally, 
\[
\avg{w_{N}}_{I}=\frac{N|F|}{|I|}.
\]
\end{proof}

\subsubsection{Computation of the Orlicz maximal function}

In our next proposition we compute the averages of the Orlicz maximal
function.
\begin{prop}
\label{prop:OrliczAux} For every interval $I$ containing some $x\in J$
such that $|F|=|I\cap E_{N}|>0$ we have that 
\[
\|w_{N}\|_{\Phi_{\alpha},I}=\frac{N}{\Phi^{-1}_{\alpha}\left(\frac{|I|}{|F|}\right)}.
\]
Furthermore, for $|F|$ small enough, namely $\frac{|I|}{|F|}$ large
enough, 
\[
\|w_{N}\|_{\Phi_{\alpha},I}\simeq_{\alpha}\frac{N|F|}{|I|}A_{\alpha}\left(\frac{|I|}{|F|}\right).
\]
\end{prop}

\begin{proof}
We begin noting that $w_{N}=N\chi_{F}$. Hence 
\[
\frac{1}{|I|}\int_{I}\Phi_{\alpha}\!\left(\frac{w_{N}}{\lambda}\right)\,dx=\frac{|F|}{|I|}\Phi_{\alpha}\!\left(\frac{N}{\lambda}\right).
\]
Then, since $\Phi_{\alpha}$ is increasing, 
\[
\frac{1}{|I|}\int_{I}\Phi_{\alpha}\!\left(\frac{w_{N}}{\lambda}\right)\,dx\leq1\iff\Phi_{\alpha}\!\left(\frac{N}{\lambda}\right)\leq\frac{|I|}{|F|}\iff\lambda\geq\frac{N}{\Phi^{-1}_{\alpha}\left(\frac{|I|}{|F|}\right)}
\]
Taking infimum over $\lambda>0$ this implies $\|w_{N}\|_{\Phi_{\alpha},I}=\frac{N}{\Phi^{-1}_{\alpha}\left(\frac{|I|}{|F|}\right)}.$
The remaining assertion follows from the first one combined with Lemma
\ref{lem:inverse}.
\end{proof}

\begin{lem}
\label{lem:OrliczLem} For $N$ sufficiently large and uniformly for
$x\in J$, 
\[
M_{\Phi_{\alpha}}w_{N}(x)\simeq_{\alpha}L_{2}(N)L_{3}(N)^{\alpha}.
\]
\end{lem}

\begin{proof}
We begin with the lower bound. Let $F=Q_{0}\cap E_{N}$. For the interval
$Q_{0}=(0,1)$, which contains every $x\in J$, 
\[
|F|=|Q_{0}\cap E_{N}|=|(0,1)\cap(0,1/N)|=\frac{1}{N}
\]
proposition \ref{prop:OrliczAux} gives 
\[
M_{\Phi_{\alpha}}w_{N}(x)\geq\|w_{N}\|_{\Phi_{\alpha},Q_{0}}\simeq_{\alpha}\frac{N}{\Phi^{-1}_{\alpha}\left(\frac{|Q_{0}|}{|F|}\right)}\simeq_{\alpha}\frac{N|F|}{|Q_{0}|}A_{\alpha}\left(\frac{|Q_{0}|}{|F|}\right)=A_{\alpha}\left(N\right)
\]
for $N$ large enough.

For the upper bound, we take any interval $I\ni x$ with $|F|=|I\cap E_{N}|>0$.
By Proposition \ref{prop:Trivialstuff} $\frac{|I|}{|F|}\geq\frac{3N}{8}$
for $N\geq8$. Taking this and Lemma~\ref{lem:slow} with exponent
$\eta=1/2$ into account,
\[
A_{\alpha}\left(\frac{|I|}{|F|}\right)\lesssim_{\alpha}A_{\alpha}\left(\frac{3}{8}N\right)\left(\frac{\frac{|I|}{|F|}}{N}\right)^{1/2}\simeq_{\alpha}A_{\alpha}(N)\left(\frac{|I|}{N|F|}\right)^{1/2}.
\]
Therefore by \ref{prop:OrliczAux},
\begin{align*}
\|w_{N}\|_{\Phi_{\alpha},I} & \lesssim_{\alpha}\frac{N|F|}{|I|}A_{\alpha}\left(\frac{|I|}{|F|}\right)\simeq_{\alpha}\frac{N|F|}{|I|}A_{\alpha}(N)\left(\frac{|I|}{N|F|}\right)^{1/2}\\
 & =\left(\frac{N|F|}{|I|}\right)^{\frac{1}{2}}A_{\alpha}(N)\lesssim A_{\alpha}(N),
\end{align*}
where we have used again that $\frac{|I|}{|F|}\geq\frac{3N}{8}$.
As we mentioned before Proposition \ref{prop:Trivialstuff} intervals
with $|I\cap E_{N}|=0$ do not contribute in the computation of $M_{\Phi_{\alpha}}w_{N}.$
Hence, taking supremum over the intervals such that $|I\cap E_{N}|>0$,
we have that 
\[
M_{\Phi_{\alpha}}w_{N}(x)\lesssim_{\alpha}A_{\alpha}(N).
\]
\end{proof}

\subsubsection{Computation of the Entropy maximal function}
\begin{lem}
\label{Lem:EntropyMaxLem} For every fixed $\delta>0$, uniformly
for $x\in J$, 
\[
M_{\varepsilon_{\delta}}w_{N}(x)\simeq_{\delta}L_{2}(N)L_{3}(N)L_{4}(N)^{1+\delta}.
\]
\end{lem}

\begin{proof}
We begin recalling that 
\[
M_{\varepsilon_{\delta}}w_{N}(x)=\sup_{x\in I}\avg{w_{N}}_{I}\,L_{1}(\rho_{w_{N}}(I))\,\varepsilon_{\delta}(\rho_{w_{N}}(I))=\sup_{x\in I}\avg{w_{N}}_{I}B_{\delta}(\rho_{w_{N}}(I))
\]
where
\[
B_{\delta}(s)=L_{1}(s)\varepsilon_{\delta}(s)=L_{1}(s)L_{2}(s)L_{3}(s)^{1+\delta}
\]
Let $x\in J$. Assume again that $|E_{N}\cap I|=|F|>0$ and that $x\in I$.
Observe that Lemma~\ref{lem:lemIndicator} gives 
\[
\rho_{w_{N}}(I)\simeq1+\log\frac{|I|}{|F|}.
\]
By monotonicity of $B_{\delta}$ and Lemma~\ref{lem:shift}, 
\[
B_{\delta}(\rho_{w_{N}}(I))\lesssim_{\delta}H_{\delta}\left(\frac{|I|}{|F|}\right).
\]
Combining this with 
\[
\avg{w_{N}}_{I}B_{\delta}(\rho_{w_{N}}(I))\lesssim_{\delta}\frac{N|F|}{|I|}H_{\delta}\left(\frac{|I|}{|F|}\right).
\]
Since by Proposition \ref{prop:Trivialstuff} $\frac{|I|}{|F|}\geq\frac{3N}{8}$,
Lemma~\ref{lem:slow}, with $\eta=1/2$, gives for $N$ large enough
that
\[
H_{\delta}\left(\frac{|I|}{|F|}\right)\lesssim_{\delta}H_{\delta}(N)\left(\frac{|I|}{|F|N}\right)^{1/2}.
\]
Gathering the two estimates above, 
\[
\avg{w_{N}}_{I}B_{\delta}(\rho_{w_{N}}(I))\lesssim_{\delta}H_{\delta}(N)\left(\frac{N|F|}{|I|}\right)^{1/2}\lesssim_{\delta}H_{\delta}(N).
\]
Consequently 
\begin{equation}
M_{\varepsilon_{\delta}}w_{N}(x)\lesssim_{\delta}H_{\delta}(N).\label{eq:UpperMEnt}
\end{equation}
For the reverse inequality, again choose $Q_{0}=(0,1)$. Here 
\[
\avg{w_{N}}_{Q_{0}}=1,\qquad\frac{|Q_{0}|}{|E_{N}\cap Q_{0}|}=N.
\]
As above, by Lemma \ref{lem:lemIndicator} 
\[
\rho_{w_{N}}(Q_{0})\simeq1+\log N,
\]
and by Lemma \ref{lem:shift},
\[
B_{\delta}(\rho_{w_{N}}(Q_{0}))\simeq_{\delta}H_{\delta}(N).
\]
Consequently 
\begin{equation}
M_{\varepsilon_{\delta}}w_{N}(x)\geq\avg{w_{N}}_{Q_{0}}B_{\delta}(\rho_{w_{N}}(Q_{0}))\gtrsim_{\delta}H_{\delta}(N).\label{eq:lowerMEnt}
\end{equation}
Gathering \eqref{eq:UpperMEnt} and \eqref{eq:lowerMEnt} ends the
proof.
\end{proof}

\section{Transference principle and consequences }\label{sec:TP+corollaries}

\subsection{The sparse transference principle}

We begin our proof with the following lemma.
\begin{lem}
\label{lem:maxprinciple} Let $\mathcal{S}$ be a finite family in
one dyadic lattice and let $g\geq0$. For $a>0$, set 
\[
\Omega_{a}:=\{\AS g>a\},\qquad F_{a}:=\Omega^{c}_{a}.
\]
Then 
\begin{equation}
\AS(g\chi_{F_{a}})(x)\leq a\qquad\text{a.e }x\in\R^{n}.\label{eq:maxprinciple}
\end{equation}
\end{lem}

\begin{proof}
For $x\in F_{a}$ , positivity gives $\AS(g\chi_{F_{a}})(x)\leq\AS g(x)\leq a$.
Now let $x\in\Omega_{a}$ away from the union of the dyadic boundaries,
and let $P$ be the maximal dyadic cube contained in $\Omega_{a}$
and containing $x$. Such a cube exists for $A_{\mathcal{S}}g$ due
to the fact that it is locally constant away from the union of dyadic
boundaries. Note that if $Q\subset P$, since $Q\cap F_{a}=\emptyset$,
we have that 
\[
\avg{g\chi_{F_{a}}}_{Q}=0.
\]
Any dyadic cube $Q$ which contains $x$ and is not contained in $P$
must strictly contain $P$. Since $P$ is maximal, there exists $z\in\tilde{P}\cap F_{a}$,
where $\tilde{P}$ is the dyadic parent of $P$. Every dyadic ancestor
of $P$ contains $z$. Hence 
\begin{align*}
\AS(g\chi_{F_{a}})(x) & =\sum_{\substack{Q\in\mathcal{S}\\
Q\supsetneq P
}
}\avg{g\chi_{F_{a}}}_{Q}=\sum_{\substack{Q\in\mathcal{S}\\
Q\supsetneq P
}
}\avg{g\chi_{F_{a}}}_{Q}\chi_{Q}(z)\\
 & \leq\sum_{\substack{Q\in\mathcal{S}\\
Q\supsetneq P
}
}\avg g_{Q}\chi_{Q}(z)\leq\AS g(z)\leq a.
\end{align*}
This proves the assertion. 

Our next Theorem contains the announced transference principle.
\end{proof}

\begin{thm}
\label{thm:transferSparse} Let $\mathcal{S}$ be a dyadic sparse
family, $w\geq0$ a weight and $u>0$ be a measurable function $w$-almost
everywhere. Assume that for every $g\geq0$, 
\begin{equation}
\sup_{t>0}t\,w(\{\AS g>t\})\leq\kappa\int_{\R^{n}}gu.\label{eq:abstractprimal}
\end{equation}
Then, for every $f\geq0$, 
\begin{equation}
\sup_{\lambda>0}\lambda\,w\!\left(\left\{ \frac{\AS f}{u}>\lambda\right\} \right)\leq4\kappa\int_{\R^{n}}f.\label{eq:abstractdual}
\end{equation}
\end{thm}

\begin{proof}
Assume first that $\mathcal{S}$ is finite and let
\[
E=\{\AS f>\lambda u\}.
\]
Now we define 
\[
E_{N}=E\cap B(0,N)\cap\left\{ \frac{1}{N}\leq u\leq N\right\} .
\]
Clearly that set has finite $w$-measure. Put 
\[
g_{N}=\frac{w\chi_{E_{N}}}{u},\qquad a=2\kappa,\qquad\Omega_{N}=\{\AS g_{N}>a\},\qquad F_{N}=\Omega^{c}_{N}.
\]
Clearly $g_{N}$ is locally integrable. Since $\int g_{N}u=w(E_{N})$,
hypothesis \eqref{eq:abstractprimal} gives 
\[
w(\Omega_{N})\leq\frac{\kappa}{a}w(E_{N})=\frac{1}{2}w(E_{N}).
\]
Note that then if we call $E'_{N}=E_{N}\cap F_{N}$,
\[
w(E_{N})=w(\Omega_{N}\cap E_{N})+w(E'_{N})\leq\frac{1}{2}w(E_{N})+w(E'_{N}).
\]
Hence 
\[
w(E'_{N})\geq\frac{1}{2}w(E_{N}).
\]

Now we continue our argument as follows. By the definition of $E'_{N}$
and the self-adjointness of $\AS$, 
\begin{align*}
\lambda w(E'_{N}) & \leq\int_{E'_{N}}\AS f(x)\frac{w(x)}{u(x)}\dd x\\
 & =\int_{\R^{n}}f(x)\AS\!\left(\frac{w\chi_{E'_{N}}}{u}\right)(x)\dd x\\
 & =\int_{\R^{n}}f(x)\AS(g_{N}\chi_{F_{N}})(x)\dd x.
\end{align*}
Now we apply Lemma \ref{lem:maxprinciple} and we have that
\[
\int_{\R^{n}}f(x)\AS(g_{N}\chi_{F_{N}})(x)\dd x\leq a\int_{\R^{n}}f(x)\dd x.
\]
Gathering the estimates above
\[
w(E'_{N})\leq a\int_{\R^{n}}\frac{f(x)}{\lambda}\dd x
\]
Since $w(E'_{N})\geq w(E_{N})/2$ and $a=2\kappa$, we conclude that
\[
w(E_{N})\le2w(E_{N}')\leq4\kappa\int_{\R^{n}}\frac{f(x)}{\lambda}\dd x
\]
and letting $N\rightarrow\infty$ allows us to conclude by monotone
convergence that 
\[
w(\{\AS f>\lambda u\})=w(E)\leq4\kappa\int_{\R^{n}}\frac{f(x)}{\lambda}\dd x
\]

Finally, note that for a countable sparse family $\mathcal{S}$, we
may consider an increasing sequence of sparse families $\mathcal{S}_{m}\subset\mathcal{S}$
such that $A_{\mathcal{S}_{m}}f\rightarrow A_{\mathcal{S}}f$. Note
that since $A_{\mathcal{S}_{m}}f\leq A_{\mathcal{S}}f$ the hypothesis
\eqref{eq:abstractprimal} remains true for the same constant $\kappa$.
Applying the finite families case we have just established, and then
passing to the limit by monotone convergence of the increasing with
$m$ level sets completes the proof.
\end{proof}

\subsection{Proof of Theorem \ref{thm:Transf}}

We begin recalling that if $T$ is a Calderón-Zygmund or a maximal
Calderón-Zygmund operator then there exist $3^{n}$ dyadic lattices
$\mathcal{D}_{j}$ and $3^{n}$ $\eta$-sparse families $\mathcal{S}_{j}\subset\mathcal{D}_{j}$
such that 
\begin{equation}
|Tf(x)|\lesssim\sum^{3^{n}}_{j=1}A_{\mathcal{S}_{j}}(|f|).\label{eq:sparsecontrol}
\end{equation}
We remit the reader, among other references, to \cite{LN,L} for the
proof of this result. Observe that as it was shown in \cite{LN} we
may change the sparse constant for a larger one just changing the
number of sparse operators involved in case of need. Hence we do not
take care about the constant. Note that then, for a pair of weights
$u$ and $w$ in the conditions of the hypothesis, since 
\[
\|A_{\mathcal{S}}(|f|)\|_{L^{1,\infty}(w)}\leq\kappa_{u,w,\eta}\|f\|_{L^{1}(u)}
\]
uniformly on $\mathcal{S}$, we have that 
\[
\left\Vert \frac{A_{\mathcal{S}}(|f|)}{u}\right\Vert _{L^{1,\infty}(w)}\leq4\kappa_{u,w,\eta}\|f\|_{L^{1}}
\]
Combining this estimate with \eqref{eq:sparsecontrol} leads to the
desired conclusion.

\subsection{Corollaries of the transference principle }\label{subsec:CorollariesTP}

Our corollaries will allow us to derive some new dual Muckenhoupt-Wheeden
estimates from known sparse uniform estimates in the literature as
we announced in Section \ref{sec:IMR}. We begin recalling the following
estimate.
\begin{lem}[\cite{DSLR}]
et $\Phi$ be a Young function such that $c_{\Phi}=\sum^{\infty}_{k=1}\frac{1}{\overline{\Phi}^{-1}(2^{2^{k}})}<\infty$.
For every sparse operator $A_{\mathcal{S}},$ where $\mathcal{S}$
is a $\eta$-sparse family,
\[
\|A_{\mathcal{S}}f\|_{L^{1,\infty}(w)}\lesssim_{\eta}\bigl(1+c_{\Phi}\bigr)\|f\|_{L^{1}(M_{\Phi}w)}.
\]
and the constant is uniform on $\mathcal{S}$ for a fixed parameter
$\eta$.
\end{lem}

Combining this estimate with Theorem \ref{thm:Transf} we obtain the
following Corollary.
\begin{cor}
Let $\Phi$ be a Young function such that $c_{\Phi}=\sum^{\infty}_{k=1}\frac{1}{\overline{\Phi}^{-1}(2^{2^{k}})}<\infty$.
If $T$ is a Calderón-Zygmund operator, 
\[
\left\Vert \frac{Tf}{M_{\Phi}w}\right\Vert _{L^{1,\infty}(w)}\lesssim(1+c_{\Phi})\|f\|_{L^{1}}.
\]
\end{cor}

Note that the result above works choosing in particular, for $\rho>0$,
\begin{align*}
\Phi(t) & =t\log(e+t)^{\rho}\\
\Phi(t) & =t\log(e+\log(e+t))^{1+\rho}\\
\Phi(t) & =t\log(e+\log(e+t))\log(e+\log(e+\log(e+t)))^{1+\rho}.
\end{align*}
 In all those cases, if additionally $\rho\leq1$, then $c_{\Phi}\simeq\frac{1}{\rho}$.

Now we revisit the following estimate. 
\begin{lem}[\cite{Ra}]
Let $\varepsilon:[1,\infty)\rightarrow[1,\infty)$ be an increasing
function such that $\kappa_{\varepsilon}=\sum^{\infty}_{k=-1}\frac{1}{\varepsilon(2^{2^{k}})}<\infty$.
For every sparse operator $A_{\mathcal{S}},$ where $\mathcal{S}$
is a $\eta$-sparse family,
\[
\|A_{\mathcal{S}}f\|_{L^{1,\infty}(w)}\lesssim_{\eta}\kappa_{\varepsilon}\|f\|_{L^{1}(M_{\varepsilon}w)}.
\]
and the constant is uniform on $\mathcal{S}$ for a fixed parameter
$\eta$.
\end{lem}

Again, combining this result with Theorem \ref{thm:Transf} yields
the following Corollary. 
\begin{cor}
Let $\varepsilon:[1,\infty)\rightarrow[1,\infty)$ be an increasing
function such that $\kappa_{\varepsilon}=\sum^{\infty}_{k=-1}\frac{1}{\varepsilon(2^{2^{k}})}<\infty$.
If $T$ is a Calderón-Zygmund operator, 
\[
\left\Vert \frac{Tf}{M_{\varepsilon}w}\right\Vert _{L^{1,\infty}(w)}\lesssim\kappa_{\varepsilon}\|f\|_{L^{1}}.
\]
\end{cor}

This result can be applied, for instance, to the case
\[
\varepsilon(t)=\log(e+\log(e+t))\log(e+\log(e+\log(e+t)))^{1+\rho}\qquad\rho>0.
\]
In that case, if additionally $\rho\leq1$, $\kappa_{\varepsilon}\simeq\frac{1}{\rho}$.
Note that up until now no dual Muckenhoupt-Wheeden estimate had been
obtained for Rahm entropy maximal functions.

\section*{Declaration of use of AI}

The ideas for the main results were obtained in a series of chats
with ChatGPT 5.6. The author verified, and rewrote all parts of the
paper influenced by AI-generated material. 

\section*{Acknowledgement}

The author was partially supported by the Spanish Ministry of Science
and Innovation through the project PID2022-136619NB-I00 funded by
MCIN/AEI/10.13039/501100011033/FEDER, UE and by Junta de Andalucía
through the project FQM-354


\begin{thebibliography}{10}
\bibitem{CLO} M.~Caldarelli, A.~K. Lerner and S.~Ombrosi, \emph{On
a counterexample related to weighted weak type estimates for singular
integrals}, Proc. Amer. Math. Soc. \textbf{145} (2017), no.~7, 3005--3012.

\bibitem{CO} M.~Cao and A.~Olivo, \emph{Two-weight extrapolation
on function spaces and applications}, Math. Nachr. \textbf{297} (2024),
2399--2444.

\bibitem{CUP}D. Cruz-Uribe, C. Pérez, \emph{Two-weight extrapolation
via the maximal operator}. J. Funct. Anal. \textbf{174} (2000) no.
1, 1--17.

\bibitem{DSLR} C.~Domingo-Salazar, M.~T. Lacey and G.~Rey, \emph{Borderline
weak type estimates for singular integrals and square functions},
Bull. Lond. Math. Soc. \textbf{48} (2016), no.~1, 63--73.

\bibitem{FS}C. Fefferman, E. M. Stein, \emph{Some maximal inequalities}
Amer. J. Math. \textbf{93} (1971), 107--115.

\bibitem{H}T. P. Hytönen, \emph{The sharp weighted bound for general
Calderón-Zygmund operators.} Ann. of Math. (2) \textbf{175} (2012),
no. 3, 1473--1506.

\bibitem{LS}M.T. Lacey, S. Spencer, \emph{On entropy bumps for Calderón-Zygmund
operators} Concr. Oper. \textbf{2} (2015), no. 1, 47--52.

\bibitem{L}A.K. Lerner, \emph{On pointwise estimates involving sparse
operators} New York J. Math. \textbf{22} (2016), 341--349.

\bibitem{LN}A.K. Lerner, F. Nazarov, \emph{Intuitive dyadic calculus:
the basics}, Expo. Math. \textbf{37} (2019), no. 3, 225--265.

\bibitem{LNO}A.~K. Lerner, F. Nazarov, S.~Ombrosi, \emph{On the
sharp upper bound related to the weak Muckenhoupt-Wheeden conjecture}
Anal. PDE \textbf{13} (2020), no. 6, 1939--1954.

\bibitem{LOP} A.~K. Lerner, S.~Ombrosi and C.~Pérez, \emph{Weak
type estimates for singular integrals related to a dual problem of
}\emph{Muckenhoupt}\emph{--Wheeden}, J. Fourier Anal. Appl. \textbf{15}
(2009), 394--403.

\bibitem{LOPA1}A.~K. Lerner, S.~Ombrosi and C.~Pérez,\emph{ $A_{1}$
bounds for Calderón-Zygmund operators related to a problem of }\emph{Muckenhoupt}\emph{
and }\emph{Wheeden}. Math. Res. Lett. \textbf{16} (2009), no. 1, 149--156.

\bibitem{O}A. Osękowski, \emph{The dual conjecture of }\emph{Muckenhoupt}\emph{
and }\emph{Wheeden}. Rev. Mat. Iberoam. 37 (2021), no. 6, 2285--2308.

\bibitem{P}C.~Pérez, \emph{Weighted norm inequalities for singular
integral operators} J. London Math. Soc. (2) \textbf{49} (1994), no.
2, 296--308.

\bibitem{Ra}R. Rahm, \emph{Borderline weak-type estimates for sparse
bilinear forms involving $A_{\infty}$ maximal functions}, J. Math.
Anal. Appl. \textbf{504} (2021), no. 1, Paper No. 125372, 10 pp.

\bibitem{RS1}R. Rahm, S. Spencer, \emph{Entropy bump conditions for
fractional maximal and integral operators} Concr. Oper. \textbf{3}
(2016), no. 1, 112--121.

\bibitem{RS2}R. Rahm, S. Spencer, \emph{Entropy bumps and another
sufficient condition for the two-weight boundedness of sparse operators}
Israel J. Math. \textbf{223} (2018), no. 1, 197--204.

\bibitem{R}M. C.\emph{ }Reguera,\emph{ On Muckenhoupt-Wheeden conjecture},
Adv. Math. \textbf{227} (2011), no. 4, 1436--1450.

\bibitem{RT}M. C.\emph{ }Reguera, C. Thiele, \emph{The Hilbert transform
does not map $L^{1}(Mw)$ to $L^{1,\infty}(w)$} Math. Res. Lett.
\textbf{19} (2012), no. 1, 1--7.

\bibitem{TV}S. Treil, A. Volberg, \emph{Entropy conditions in two
weight inequalities for singular integral operators}. Adv. Math. \textbf{301}
(2016), 499--548.

\end{thebibliography}
\end{document}